\documentclass[preprint,review,10pt]{amsart}
\usepackage{amsmath,amssymb,amsfonts,xspace,mathrsfs}
\newtheorem{theorem}{Theorem}[section]

\theoremstyle{definition}
\newtheorem{definition}[theorem]{Definition}
\newtheorem{example}[theorem]{Example}

\newtheorem{proposition}[theorem]{Proposition}

\theoremstyle{remark}

\numberwithin{equation}{section}

\begin{document}

\title{ On  $n$-ary $S$-hyperideals }

\author{Mahdi Anbarloei}
\address{Department of Mathematics, Faculty of Sciences,
Imam Khomeini International University, Qazvin, Iran.
}

\email{m.anbarloei@sci.ikiu.ac.ir}





\begin{abstract}
 In this paper, we introduce and study the notion of  $n$-ary $S$-hyperideals in a Krasner $(m,n)$-hyperring. 

\end{abstract}
\maketitle

\section{Introduction} 
In commutative algebra, multiplicatively closed subsets are essential tools for the localization of rings and modules. Beyond this, they constitute a core element in    ideal theory, providing an adaptable apparatus for exploring and categorizing various ideal classes. The prominent concept of $S$-ideals has     recently been introduced via a multiplicatively closed subset.   Assume that $I$ is a proper ideal of a commutative ring $R$, and   $S$  is a multiplicatively closed subset of $R$.   $I$ is called  an $S$-ideal if for all $x, y \in R$,   $xy \in I$ and $x \in S$ imply that $y \in I$  \cite{Khashan}.

Parallel to the evolution of classical abstract algebra, the framework of algebraic hyperstructures arose, substantially generalizing standard algebraic ideas. Hyperstructure theory, which traces its origins back to F. Marty in 1934 \cite{s1}, extends conventional algebra by defining multi-valued operations where the composition of two elements  yields a set rather than a single element. Over the decades, this idea has generalized to encompass a diverse variety of structures, including hyperrings, hyperlattices, hyperfields, and hypermodules \cite{s3, davvaz1, Konstantindou, s4, s10}. An important advancement in this domain was the notion of $n$-ary hyperstructures, which involve hyperoperations with multiple inputs, thereby providing a more flexible algebraic framework \cite{s9,l1,l2}. In particular, Mirvakili and Davvaz \cite{d1} introduced Krasner $(m, n)$-hyperrings, which combine the properties of hyperrings with an $m$-ary hyperoperation and an $n$-ary operation. Several key concepts of hyperideals, such as $n$-ary prime, $n$-ary primary, maximal, $(k,n)$-absorbing (primary) hyperideals, and   weakly $(k,n)$-absorbing (primary) hyperideals, were introduced and subsequently analyzed in \cite{sorc1,davvazz,rev2}. The notion of $n$-ary multiplicative subsets of a Krasner $(m,n)$-hyperring  was introduced in \cite{sorc1}, serving as a highly applicable tool in the structure theory of hyperideals. 

Let $S$ be an $n$-ary multiplicative subset of a Krasner $(m,n)$-hyperring $\mathscr{A}$.
In this paper, we introduce and investigate the notion of $n$-ary $S$-hyperideals, which  offer a broad perspective on the classification of hyperideals by means of  $S$. Among the various results established  herein,  we obtain that the radical of every $n$-ary $S$-hyperideal of $\mathscr{A}$ is also an $n$-ary $S$-hyperideal in Theorem \ref{1}(2).
In Theorem \ref{6}, we conclude that any minimal $n$-ary prime hyperideal over an $n$-ary $S$-hyperideal is itself an $n$-ary $S$-hyperideal of $\mathscr{A}$.
In Theorem \ref{4}, we determine the smallest $n$-ary $S$-hyperideal of $\mathscr{A}$ containing an arbitrary hyperideal.
Proposition \ref{3} shows that for every hyperideal $P$ of $\mathscr{A}$, there exists  an $n$-ary multiplicative subset $S$ of $\mathscr{A}$ such that $P$ becomes an $n$-ary $S$-hyperideal. 
Proposition \ref{intersection} verifies   that the class of $n$-ary $S$-hyperideals is closed under arbitrary intersections.
Although the notions of $S$-hyperideals and primary hyperideals are distinct, Theorem \ref{ejavad} establishes the equivalence between $n$-ary $S$-hyperideals and $n$-ary $Q$-primary hyperideals when $S$ is chosen as the complement of the minimal prime hyperideal $Q$.
We obtain that if $S$ is the set of all elements outside the finite union of the minimal prime hyperideals, then every $S$-hyperideal has a primary decomposition in Theorem \ref{complement}.
An $S$-hyperideal version of the Prime Avoidance Lemma is studied in Theorem \ref{avoidance}. Finally, we explore how $n$-ary $S$-hyperideals behave under various hyperring-theoretic constructions.
\section{Preliminaries}
This section is devoted to recalling some fundamental definitions concerning $n$-ary hyperstructures, which will be utilized throughout this work.

Given a non-empty set $\mathscr{A}$, we denote by $\mathcal{P}^*(\mathscr{A})$ the set of its non-empty subsets. An $n$-ary hyperoperation on $\mathscr{A}$ is defined as a mapping $f : \mathscr{A}^n \to \mathcal{P}^*(\mathscr{A})$, in which case the algebraic system $(\mathscr{A}, f)$ constitutes an $n$-ary hypergroupoid.
To simplify the notation, the sequence $p_i, p_{i+1}, \dots, p_j$ is abbreviated as $p^j_i$, with the understanding that this symbol is empty whenever $j < i$. Consequently, the expression $f(p_1, \dots, p_i, q_{i+1}, \dots, q_j, r_{j+1}, \dots, r_n)$ is rendered as $f(p^i_1, q^j_{i+1}, r^n_{j+1})$. In the special case where $q_{i+1} = \cdots = q_j = q$, this further reduces to $f(p^i_1, q^{(j-i)}, r^n_{j+1})$. Moreover, we  define $f(P^n_1)= \bigcup \{f(p^n_1) \ \vert \ p_i \in P_i, 1 \leq i \leq n \}$ where $P_1,\cdots, P_n$ are non-empty subsets of $\mathscr{A}$.
\begin{definition}
\cite{d1} An algebraic hyperstructure  $(\mathscr{A}, f, g)$, or simply $\mathscr{A}$, is called a commutative Krasner $(m, n)$-hyperring with a scalar identity $1_\mathscr{A}$ if  it meets the following conditions:

\begin{itemize} 
\item[\rm{{\bf i.}}]~ $(\mathscr{A}, f$) is a canonical $m$-ary hypergroup, 

\item[\rm{{\bf ii.}}]~ $(\mathscr{A}, g)$ is a $n$-ary semigroup, 

\item[\rm{{\bf iii.}}]~ $g(p^{i-1}_1, f(q^m _1 ), p^n _{i+1}) = f(g(p^{i-1}_1, q_1,p^n_{ i+1}),\ldots, g(p^{i-1}_1, q_m, p^n_{ i+1}))$ for all   $1 \leq i \leq n$ and $p^{i-1}_1 , p^n_{ i+1}, q^m_ 1 \in \mathscr{A}$. 

\item[\rm{{\bf iv.}}]~ $g(0, p^n _2) = g(p_2, 0, p^n _3) = \cdots = g(p^n_ 2, 0) = 0$ for all $p^n_ 2 \in \mathscr{A}$, 

\item[\rm{{\bf v.}}]~ $g(p_1^n) = g(p_{\sigma(1)}^{\sigma(n)})$ for all $p_1^n \in \mathscr{A}^n$ and any permutation $\sigma \in S_n$, 

\item[\rm{{\bf vi.}}]~ $g(p,1_\mathscr{A}^{(n-2)})=p$ for all $p \in \mathscr{A}$.
\end{itemize}
\end{definition}
Throughout this paper, $\mathscr{A}$ denotes a commutative Krasner $(m,n)$-hyperring with a scalar identity $1_\mathscr{A}$.
\begin{definition}
Let $P$ be a non-empty subset of $(\mathscr{A},f,g)$.  
\begin{itemize} 
\item[\rm{{\bf 1.}}]~ $P$ is called a subhyperring of $\mathscr{A}$ if $(P, f, g)$ is a Krasner $(m, n)$-hyperring. 
\item[\rm{{\bf 2.}}]~ $P$ is called a hyperideal of $\mathscr{A}$ if $(P, f)$ is an $m$-ary subhypergroup
of $(\mathscr{A}, f)$ and $g(p^{i-1}_1, P, p_{i+1}^n) \subseteq P$, for every $p^n _1 \in \mathscr{A}$ and $1 \leq i \leq n$.
\end{itemize}
\end{definition}

\begin{definition} \cite{sorc1} The hyperideal generated by an element $p$ of $\mathscr{A}$, denoted by $\langle p \rangle$,  is defined by $\langle p \rangle=\{g(r,p,1_\mathscr{A}^{(n-2)}) \ \vert \ r \in \mathscr{A}\}.$
\end{definition}
\begin{definition} \cite{sorc1}
A hyperideal $M$ of $\mathscr{A}$
is  maximal if for each hyperideal $N$ of $\mathscr{A}$, $M \subseteq N \subseteq \mathscr{A}$ implies that $N=M$ or $N=\mathscr{A}$. The intersection of all maximal hyperideals of $\mathscr{A}$ is denoted by $J_{(m,n)}(\mathscr{A})$.
\end{definition}

\begin{definition}
\cite{sorc1} A proper hyperideal  $P$ of $\mathscr{A}$ is called an $n$-ary  prime hyperideal if for hyperideals $P_1^n$ of $\mathscr{A}$, $g(P_1^ n) \subseteq P$ implies that $P_i \subseteq P$ for some $1 \leq i \leq n$.
\end{definition}
 Lemma 4.5 in \cite{sorc1} indicates that  a proper hyperideal $P$ of $\mathscr{A}$ is $n$-ary  prime if  $g(p^n_ 1) \in P$ and  $p^n_ 1 \in A$ imply that $p_i \in P$ for some $1 \leq i \leq n$. 

\begin{definition} \cite{sorc1} 
Let $P$ be a hyperideal of $\mathscr{A}$. The intersection of  all $n$-ary prime hyperideals $\mathscr{A}$  which contain   $P$ is called radical of $P$ and it is  denoted by ${\bf r} (P)$. If the set of all $n$-ary prime hyperideals   containing $P$ is empty, then ${\bf r}(P)= \mathscr{A}$.
\end{definition}
In \cite{sorc1}, it was shown  that if $p \in {\bf r}(P)$, then 
there exists $w \in \mathbb {N}$ with $g(p^ {(w)} , 1_\mathscr{A}^{(n-w)} ) \in P$ for $w \leq n$, or $g_{(l)} (p^ {(w)} ) \in P$ for $w = l(n-1) + 1$.

\begin{definition} \cite{sorc1}
Assume that $S \subseteq \mathscr{A}$ is  non-empty. We say that $S$ is  an  $n$-ary multiplicative subset (MS) of $\mathscr{A}$  if $g(s_1^n) \in S$ for every $s_1^n \in S$.
\end{definition}

\section{$S$-hyperideals}
In this section , we establish the foundations of $n$-ary $S$-hyperideals. We strat with the following definition.
\begin{definition}
Let $S$ be an MS of $\mathscr{A}$. A proper hyperideal $P$ of $\mathscr{A}$ refers to an $n$-ary $S$-hyperideal if  $g(p_1^n) \in P$ for   $p_1^n \in  \mathscr{A}$ and $p_i \in S$ for some $1 \leq i \leq n$ imply that $g(p_1^{i-1},1_\mathscr{A},p_{i+1}^n) \in P$. 
\end{definition}

\begin{example} \label{ex} 
Let  $\mathscr{A}=\{0,1,2\}$. Consider  the  3-ary hyperoperation $f$ and 3-ary operation $g$   defined on $\mathscr{A}$ as

$f(0,0,0)=0, \ \ \ f(0,0,2)=2, \ \ \ f(0,1,1)=1, \ \ \ f(1,1,1)=1, \ \ f(2,2,2)=2,$

$ f(0,0,1)=1,\ \ \ f(0,2,2)=2,\ \ \ f(1,1,2)=f(1,2,2)=f(0,1,2)=A,$\\
and

$ g(1,1,1)=1,$

$g(1,1,2)=g(1,2,2)=g(2,2,2)=2,$

 $g(0,p_1^2)=0$ for $p_1^2 \in \mathscr{A}$.\\
By Example 4.10 in  \cite{sorc1}, $(\mathscr{A},\oplus,g)$ is a Krasner $(3,3)$-hyperring. In this hyperring, the set $S=\{2\}$ is an MS and $P=\{0,2\}$ is not a 3-ary $S$-hyperideal of $A$ since $g(1^{(2)},2) \in P$  and $2 \in S$ but $g(1,1,1)=1\notin P$.
\end{example}
\begin{theorem} \label{1}
Let $S$ be an MS of  $\mathscr{A}$ and let $P$ be an $n$-ary $S$-hyperideal of  $\mathscr{A}$.  Then, 
\begin{itemize}
\item[\rm{{\bf (1) }}]~ $P \cap S=\varnothing$.
\item[\rm{{\bf (2) }}]~${\bf r}(P)$ is an $n$-ary $S$-hyperideal of  $\mathscr{A}$. 
\item[\rm{{\bf (3) }}]~ $P_Q=\{a \in \mathscr{A} \ \vert \ g(a,Q,1_\mathscr{A}^{(n-2)}) \subseteq P\}$ is an $n$-ary $S$-hyperideal of  $\mathscr{A}$ where $Q \subseteq \mathscr{A} \backslash P$. 
\end{itemize}
\end{theorem}
\begin{proof}
(1) Let $p \in P \cap S$. Then, we get $g(p,1_\mathscr{A}^{(n-1)}) \in P$. Since $P$ is an $n$-ary $S$-hyperideal of $\mathscr{A}$ and $p \in S$, we conclude that $1_\mathscr{A}=g(1_\mathscr{A}^{(n)}) \in P$ which is impossible. 

(2) Let $g(p_1^n) \in {\bf r}(P)$ for $p_1^n \in \mathscr{A}$ and $p_i \in S$ for some $1 \leq i \leq n$. Then, 
there exists $w \in \mathbb {N}$ with $g(g(p_1^n)^ {(w)} , 1_\mathscr{A}^{(n-w)} ) \in P$ for $w \leq n$, or $g_{(l)} (g(p_1^n)^ {(w)} ) \in P$ for $w = l(n-1) + 1$. In the first case,  by associativity law we have \\

$\hspace{2.5cm}g \big( p_i^{(w)},g(p_1^{i-1},1_\mathscr{A},p_{i+1}^n)^{(w)},1_\mathscr{A}^{(n-2w)} \big) \in P$

$\hspace{1.8cm}\Longrightarrow g \big( p_i^{(w)},g(p_1^{i-1},1_\mathscr{A},p_{i+1}^n)^{(w)}, g(1_\mathscr{A}^{(n)}),1_\mathscr{A}^{(n-2w-1)} \big) \in P$

$\hspace{1.8cm}\Longrightarrow g \big( g(p_i^{(w)},1_\mathscr{A}^{(n-w)}),g(p_1^{i-1},1_\mathscr{A},p_{i+1}^n)^{(w)},1_\mathscr{A}^{(n-w-1)} \big) \in P.$\\

Since $S$ is an MS and $p_i \in S$, we obtain $g(p_i^{(w)},1_\mathscr{A}^{(n-w)}) \in S$. Therefore, we get $g(p_1^{i-1},1_\mathscr{A},p_{i+1}^n)^{(w)}, 1_\mathscr{A}^{(n-w)}) \in P$ as $P$ is an $n$-ary $S$-hyperideal of $\mathscr{A}$. This implies that $g(p_1^{i-1}, 1_\mathscr{A},p_{i+1}^n) \in {\bf r}(P)$, as needed. In the second case, by using a similar argument, one can complete the proof.

(3) Clearly, $P_Q \neq \mathscr{A}$. Assume that $g(p_1^n) \in P_Q$ for $p_1^n \in \mathscr{A}$ and $p_i \in S$ for some $1 \leq i \leq n$. Then, we have $g(g(p_1^n),q,1_\mathscr{A}^{(n-2)}) \in P$ for any $q \in Q$. So, $g( p_i, g(p_1^{i-1},q,p_{i+1}^n),1_\mathscr{A}^{(n-2)}) \in P$. Since $P$ is an $n$-ary $S$-hyperideal of $\mathscr{A}$ and $p_i \in S$, we obtain $g(g(p_1^{i-1},1_\mathscr{A},p_{i+1}^n),q,1_\mathscr{A}^{(n-2)})=g(g(p_1^{i-1},q,p_{i+1}^n),1_\mathscr{A}^{(n-1)}) \in P$ which means $g(p_1^{i-1},1_\mathscr{A},p_{i+1}^n) \in P_Q$. Thus, $P_Q$ is an $n$-ary $S$-hyperideal of $\mathscr{A}$.
\end{proof}

\begin{proposition} \label{2}
Let $S$ be an MS of $\mathscr{A}$ and $P$ be an $n$-ary prime hyperideal of $\mathscr{A}$ satisfying $P \cap S=\varnothing$. Then, $P$ is an $n$-ary $S$-hyperideal of $\mathscr{A}$. Furthermore, every hyperideal $Q$ of $\mathscr{A}$ satisfied $Q \cap S=\varnothing$ is contained in an $n$-ary prime $S$-hyperideal.
\end{proposition}
\begin{proof}
For the first assertion, assume that $g(p_1^n) \in P$ for $p_1^n \in \mathscr{A}$ and $p_i \in S$ for some $1 \leq i \leq n$. Therefore, $p_i \notin P$ by the hypothesis. Since $P$ is an $n$-ary prime hyperideal of $\mathscr{A}$ and  $g(p_i,g(p_1^{i-1},1_\mathscr{A},p_{i+1}^n),1_\mathscr{A}^{(n-2)})=g(p_1^n) \in P$, we conclude that $g(p_1^{i-1},1_\mathscr{A},p_{i+1}^n) \in P$. Thus, $P$ is an $n$-ary $S$-hyperideal of $\mathscr{A}$. For the second assertion, suppose that $Q$ is a hyperideal of  $\mathscr{A}$ with $Q \cap S=\varnothing$. By Theorem 4.21 in \cite{sorc1}, there exists an $n$-ary prime hyperideal $P$ of $\mathscr{A}$ such that $P \cap S=\varnothing $ and $Q \subseteq P$. Hence, we conclude that $Q$ is an $n$-ary $S$-hyperideal of $\mathscr{A}$ by the first assertion. Thus, the hyperideal $Q$ of $\mathscr{A}$ satisfied $Q \cap S=\varnothing$ is contained in an $n$-ary prime $S$-hyperideal.
\end{proof}
\begin{theorem} \label{7}
Let $S$ be an MS of  $\mathscr{A}$ containing $1_\mathscr{A} $ and  $P$ be a proper hyperideal  of  $\mathscr{A}$. If $P$ is a  maximal $n$-ary $S$-hyperideal of  $\mathscr{A}$, then $P$ is an $n$-ary prime hyperideal.
\end{theorem}
\begin{proof}
Let $P$ be a  maximal $n$-ary $S$-hyperideal of  $\mathscr{A}$. Assume that $g(p_1^n) \in P$ and $g(p_1^{i-1},1_\mathscr{A},p_{i+1}^n) \notin P$ for some $1 \leq i \leq n$. Therefore, $P_{g(p_1^{i-1},1_\mathscr{A},p_{i+1}^n)}=\{a \in \mathscr{A} \ \vert \ g(g(p_1^{i-1},1_\mathscr{A},p_{i+1}^n),a,1_\mathscr{A}^{(n-2)}) \in P \}$ is an $n$-ary $S$-hyperideal of $\mathscr{A}$ by Theorem \ref{1} (3). Since $P$ is a  maximal $n$-ary $S$-hyperideal of  $\mathscr{A}$ and $P \subseteq P_{g(p_1^{i-1},1_\mathscr{A},p_{i+1}^n)}$, we have $P_{g(p_1^{i-1},1_\mathscr{A},p_{i+1}^n)}=P$. Since  $p_i \in P_{g(p_1^{i-1},1_\mathscr{A},p_{i+1}^n)}$, we get $p_i \in P$. If $g(p_1^{i-1},1_\mathscr{A},p_{i+1}^n) \in P$, then by  a similar argument, one can complete the proof.
\end{proof}
Next, it is shown that  every minimal $n$-ary prime hyperideal of $\mathscr{A}$ over an $n$-ary $S$-hyperideal is an $n$-ary $S$-hyperideal.
\begin{theorem} \label{6}
Let $S$ be an MS of  $\mathscr{A}$ containing $1_\mathscr{A} $ and  $P$ be an $n$-ary $S$-hyperideal of  $\mathscr{A}$.  If $Q$ is a minimal $n$-ary prime hyperideal of  $\mathscr{A}$ over $P$, then $Q$ is an $n$-ary $S$-hyperideal of  $\mathscr{A}$.
\end{theorem}
\begin{proof}
Assume that $g(p_1^n) \in Q$ for $p_1^n \in \mathscr{A}$ and $p_i \in S$ for some $1 \leq i \leq n$. Then, there exists an element  $a \notin Q$ and a non-negative integer $k$ such that $g(g(p_1^n)^k,a,1_\mathscr{A}^{(n-k-1)}) \in P$, by Lemma 3.5 in \cite{mah2}. So, by associativity we have $g \big(g(p_i^k,1_\mathscr{A}^{n-k}), g(p_1^{i-1},1_\mathscr{A},p_{i+1}^n)^k,a, 1_\mathscr{A}^{n-k-2} \big) \in P$. Since $P$ is an $n$-ary $S$-hyperideal of  $\mathscr{A}$ and $g(p_i^k,1_\mathscr{A}^{n-k}) \in S$, we conclude that $g \big( g(p_1^{i-1},1_\mathscr{A},p_{i+1}^n)^k,a, 1_\mathscr{A}^{n-k-1} \big) \in P \subseteq Q$. This implies that $g(p_1^{i-1},1_\mathscr{A},p_{i+1}^n) \in Q$ as $Q$ is an $n$-ary prime hyperideal of  $\mathscr{A}$ and $a \notin Q$. Thus, $Q$ is an $n$-ary $S$-hyperideal of  $\mathscr{A}$.
\end{proof}

The following result verifies that for every hyperideal $P$ of $\mathscr{A}$, there exists an $n$-ary multiplicative  subset $S$ such that $P$ is an $n$-ary $S$-hyperideal.
\begin{theorem} \label{3}
Let $P$ be a hyperideal of $\mathscr{A}$. Then, 
$S=\{x \in \mathscr{A} \ \vert \ g(p_1^{i-1},x,p_{i+1}^n) \in P \Longrightarrow g(p_1^{i-1},1_\mathscr{A},p_{i+1}^n) \in P \ \text{for each} \  p_1^{i-1},p_{i+1}^n \in \mathscr{A}, 1 \leq i \leq n \}$ is the maximal MS with respect to which $P$ remains an $n$-ary $S$-hyperideal.
\end{theorem}
\begin{proof}
First, we show that $S$ is an MS of $\mathscr{A}$. Let $x_1^n \in S$. We verify that $g(x_1^n) \in S$. Assume that $g(p_1^{i-1},g(x_1^n),p_{i+1}^n) \in P$ for $p_1^{i-1},p_{i+1}^n \in \mathscr{A}$. Therefore, we have $g \big(x_1,g(p_1^{i-1},g(x_2^n,1_\mathscr{A}),p_{i+1}^n),1_\mathscr{A}^{(n-2)} \big) \in P$. Then, $g(p_1^{i-1},g(x_2^n,1_\mathscr{A}),p_{i+1}^n) \in P$ as $x_1 \in S$. From $g \big(x_2,g(p_1^{i-1},g(x_3^n,1_\mathscr{A}^{(2)}),p_{i+1}^n ) ,1_\mathscr{A}^{(n-2)} \big)\in P$, it follows that $g(p_1^{i-1},g(x_3^n,1_\mathscr{A}^{(2)}),p_{i+1}^n) \in P$ as $x_2 \in S$. By continuing the process, we obtain $g(p_1^{i-1},1_\mathscr{A},p_{i+1}^n) \in P$ which means $g(x_1^n) \in S$. Clealy, $P$ is an $n$-ary $S$-hyperideal. Now, let us assume that $S^{\prime}$ is an MS of $\mathscr{A}$ containing $S$ for which $P$ is an $n$-ary $S^{\prime}$-hyperideal. Take any $s \in S^{\prime}$. Assume that $g(p_1^{i-1},s,p_{i+1}^n) \in P$. Since $P$ is an $n$-ary $S$-hyperideal of $\mathscr{A}$ and  $s \in S^{\prime}$, we conclude that $g(p_1^{i-1},1_\mathscr{A},p_{i+1}^n) \in P$. This implies that $s \in S$ and so $S^{\prime} \subseteq S$. Thus, $S$ is the maximal MS with respect to which $P$ remains an $n$-ary $S$-hyperideal.
\end{proof}
In the following theorem,  we   present the smallest $n$-ary $S$-hyperideal of $\mathscr{A}$ containing an arbitrary hyperideal $Q$.
\begin{theorem} \label{4}
Let $S$ be an MS of $\mathscr{A}$ containing $1_\mathscr{A}$ and $Q$ be a hyperideal of $\mathscr{A}$. Then $Q^S=\{x \in \mathscr{A} \ \vert \ g(t,x,1_\mathscr{A}^{(n-2)}) \in Q\ \text{for some}  \ \ t \in S\}$   is the smallest $n$-ary $S$-hyperideal of $\mathscr{A}$ containing $Q$.
\end{theorem}
\begin{proof}
First, we show that $Q^S$ is a hyperideal of $\mathscr{A}$. Let $x_1^m \in Q^S$. Then, there exists $t_1^m \in S$ such that $g(t_i,x_i,1_\mathscr{A}^{(n-2)}) \in Q$ for all $1 \leq i \leq m$. Put $t=g(g(t_1^n),t_{n+1},\ldots, t_m)$ if $ m \geq n$ and $t=g(t_1^m, 1_\mathscr{A}^{n-m})$ if $m < n$. So, $t \in S$. From $g(t_i,x_i,1_\mathscr{A}^{(n-2)}) \in Q$, it follows that $g(t,x_i,1_\mathscr{A}^{(n-2)}) \in Q$  for all $1 \leq i \leq m$ and so $g(t,f(x_1^m),1_\mathscr{A}^{(n-2)})=f(g(t,x_1,1_\mathscr{A}^{(n-2)}),\ldots,g(t,x_m,1_\mathscr{A}^{(n-2)})) \subseteq Q$ as $Q$ is a hyperideal of $\mathscr{A}$. This implies that $f(x_1^m) \subseteq Q^S$.  Since $g(t,0,1_\mathscr{A}^{(n-2)}) \in Q$ for any $t \in S$, we get $0 \in  Q^S$. It is easy to verify $-x \in  Q^S$ for every $x \in  Q^S$. Assume that $p_1^{i-1},p_{i+1}^n \in  \mathscr{A}$ and $x \in Q^S$. Then, there exists an element $t \in S$ such that $g(t,x,1_\mathscr{A}) \in Q$. Therefore,  we obtain $g(t,g(p_1^{i-1},x,p_{i+1}^n),1_\mathscr{A}^{(n-2)})=g(p_1^{i-1},g(t,x,1_\mathscr{A}^{(n-2)}),p_{i+1}^n) \in Q$ which means $g(p_1^{i-1},x,p_{i+1}^n) \in Q^S$. Thus, $Q^S$ is a hyperideal of $\mathscr{A}$ by Lemma 3.3 in \cite{d1}. Let $g(p_1^n) \in Q^S$ for $p_1^n \in \mathscr{A}$ and $p_i \in S$ for some $1 \leq i \leq n$. Therefore, we obtain $g \big(g(t,p_i,1_\mathscr{A}^{(n-2)}),g(p_1^{i-1},1_\mathscr{A},p_{i+1}^n),1_\mathscr{A}^{(n-2)} \big)=g(t,g(p_1^n),1_\mathscr{A}^{(n-2)}) \in Q$ for some $t \in S$. Since $g(t,p_i,1_\mathscr{A}^{(n-2)}) \in S$, we conclude that $g(p_1^{i-1},1_\mathscr{A},p_{i+1}^n) \in Q^S$. This means that $Q^S$ is an $n$-ary $S$-hyperideals. Now, assume that $Q^{\prime}$ is an $n$-ary $S$-hyperideal containing $Q$. Let $x \in Q^S$. Then, we obtain $g(t,x,1_\mathscr{A}^{(n-2)}) \in Q$ for some $t \in S$. So,  $g(t,x,1_\mathscr{A}^{(n-2)}) \in Q^S$. This implies that $x=g(x,1_\mathscr{A}^{(n-1)})\in Q^S$ as $Q^S$ is an $n$-ary $S$-hyperideal of  $\mathscr{A}$  and $t \in S$. Hence, $Q^S$ is contained in  $Q^{\prime}$. Consequently, $Q^S$ is the smallest $n$-ary $S$-hyperideal of $\mathscr{A}$ containing $Q$.
\end{proof}
The next theorem gives a characterization of $n$-ary $S$-hyperideals.
\begin{theorem} \label{5}
Let $S$ be an MS of  $\mathscr{A}$ and  $P$ be a proper hyperideal of  $\mathscr{A}$.  
\begin{itemize}
\item[\rm{{\bf (1) }}]~ $P$ is an $n$-ary $S$-hyperideal of $\mathscr{A}$.
\item[\rm{{\bf (2) }}]~$P=P_t$ for any $t \in S$ where $P_t=\{a \in \mathscr{A} \ \vert \ g(a,t,1_\mathscr{A}^{(n-2)}) \in P\}$.
\item[\rm{{\bf (3) }}]~ $P=P^S$ where $P^S=\{x \in \mathscr{A} \ \vert \ g(t,x,1_\mathscr{A}^{(n-2)}) \in P\ \text{for some}  \ \ t \in S\}$.
\end{itemize}
\end{theorem}
\begin{proof}
(1) $\Longrightarrow$ (2) Let $P$ be an $n$-ary $S$-hyperideal of $\mathscr{A}$ and $t \in S$. The inclusion $P \subseteq P_t$ always holds. Assume that $x \in P_t$. So, $g(x,t,1_\mathscr{A}) \in P $. Since $P$ is an $n$-ary $S$-hyperideal of $\mathscr{A}$ and $t \in S$, we conclude that $x=g(x,1_\mathscr{A}^{(n-1)}) \in P $. Therefore, $P$ contions $P_t$ and so $P=P_t$.

(2) $\Longrightarrow$ (3) Let $P=P_t$ for any $t \in S$ and $x \in P_S$. Then, we get $g(t_0,x,1_\mathscr{A}^{(n-2)}) \in P $ for some $t_0 \in S$. This implies that $x \in P_{t_0}$. By the hypothesis, we get $x \in P$. This means that $P^S \subseteq P$ and so  $P=P^S$ as the reverse containment is obvious.

(3) $\Longrightarrow$ (1) Let $g(p_1^n) \in P$ for $p_1^n \in \mathscr{A}$ and $p_i \in S$ for aome $1 \leq i \leq n$. Therefore, we have $g(p_i,g(p_1^{i-1}, 1_\mathscr{A},p_{i+1}^n),1_\mathscr{A}^{(n-2)}) \in P$  and so $g(p_1^{i-1}, 1_\mathscr{A},p_{i+1}^n) \in P^S $. By the hypothesis, we have $g(p_1^{i-1}, 1_\mathscr{A},p_{i+1}^n) \in P $ which shows $P$ is an $n$-ary $S$-hyperideal of $\mathscr{A}$.
\end{proof}

Recall from \cite{sorc1} that  a proper hyperideal  $P$ of $\mathscr{A}$ is called $n$-ary  primary  if $g(p^n _1) \in P$ and $p_1^ n \in \mathscr{A}$ imply that $p_i \in P$ or $g(p_1^{i-1}, 1_\mathscr{A}, p_{ i+1}^n) \in {\bf r}(P)$ for some $ 1 \leq i \leq n$. Moreover, Theorem 4.28 in \cite{sorc1} verifies that the radical of the $n$-ary primary hyperideal $P$ is prime. In this case, we say that $P$ is an $n$-ary $Q$-primary hyperideal of $\mathscr{A}$ where ${\bf r} (P)=Q$. Now, we investigate the relation between $n$-ary $S$-hyperideals and $n$-ary primary hyperideals. 

The following example verifies that every $n$-ary primary hyperideal may not be an $n$-ary $S$-hyperideal.

\begin{example}
Consider the Krasner $(2,3)$-hyperring $(\mathscr{A}=[0,1],f,g )$ where the 2-ary hyperoperation $``f"$ is defined as $f(p,q) =\{\max\{p, q\}\}$ if $p \neq q$ or $[0,p]$ if $p =q.$ 
Moreover, $``g"$ is the ordinary multiplication on real numbers. Let  $S=(0,0.1)$ and $P=[0,0.5]$. Then, $P$ is a $3$-ary primary hyperideal of $\mathscr{A}$. However,  it is not $3$-ary $S$-hyperideal since $g(0.01,0.7,0.8) \in P $ and $0.01 \in S$ but $g(1,0.7.0.8) \notin P$.
\end{example}
The following result shows that the notion of an $n$-ary $S$-hyperideal is equivalent to that of an $n$-ary $Q$-primary hyperideal,  provided $S$ is the  complement of the minimal prime hyperideal $Q$.
\begin{theorem} \label{ejavad}
Let $P$ be a proper hyperideal of $\mathscr{A}$, $Q$  a  minimal $n$-ary prime hyperideal of $\mathscr{A}$ and $S=\mathscr{A} \backslash Q$. Then, $P$ is an $n$-ary $S$-hyperideal of $\mathscr{A}$ if and only if $P$ is an $n$-ary $Q$-primary hyperideal of $\mathscr{A}$.
\end{theorem}

\begin{proof}
$\Longrightarrow$ Assume that $P$ is an $n$-ary $S$-hyperideal of $\mathscr{A}$. By Theorem \ref{1} (1), we have $P \cap S=\varnothing$ which means $ {\bf r}(P)=Q$. Now, let $g(p^n _1) \in P$ for  $p_1^ n \in \mathscr{A}$ and $g(p_1^{i-1}, 1_\mathscr{A}, p_{ i+1}^n) \notin {\bf r}(P)$ for some $ 1 \leq i \leq n$. Since $P$ is an $n$-ary $S$-hyperideal of $\mathscr{A}$, $g\big(p_i,g(p_1^{i-1}, 1_\mathscr{A}, p_{ i+1}^n),1_\mathscr{A}^{(n-2)}\big) \in P$ and $g(p_1^{i-1}, 1_\mathscr{A}, p_{ i+1}^n) \in S$, we obtain $p_i=(p_i,1_\mathscr{A}^{(n-1)}) \in P$. Thus, $P$ is an $n$-ary $Q$-primary hyperideal of $\mathscr{A}$.

$\Longleftarrow$ Assume that $P$ is an $n$-ary $Q$-primary hyperideal of $\mathscr{A}$. Let $g(p^n _1) \in P$ for  $p_1^ n \in \mathscr{A}$ and $p_i \in S$ for some $ 1 \leq i \leq n$. So, $p_i \notin {\bf r}(P)$. By the hypothesis, we conclude that $g(p_1^{i-1}, 1_\mathscr{A}, p_{ i+1}^n) \in P$. Thus, $P$ is an $n$-ary $S$-hyperideal of $\mathscr{A}$.
\end{proof}
Next, we prove that every $S$-hyperideal of $\mathscr{A}$ has a primary decomposition if $S$ is formed by taking the complement of a finite union of minimal prime hyperideals.
\begin{theorem} \label{complement}
Let $Q_1^m$  be some  minimal $n$-ary prime hyperideals of $\mathscr{A}$ and $S=\mathscr{A} \backslash \bigcup_{j=1}^mQ_j$. If $P$ is an $n$-ary $S$-hyperideal of $\mathscr{A}$, then $P=\bigcap_{j=1}^mP_j$ such that $P_j=\{x \in \mathscr{A} \ \vert \ g(t_j,x,1_\mathscr{A}^{(n-2)}) \in P \ \text{ for some} \ t_j \notin Q_j\}$ for each $1 \leq j \leq m$.
\end{theorem}
\begin{proof}
Let $S=\mathscr{A} \backslash \bigcup_{j=1}^mQ_j$ and $P$ be an $n$-ary $S$-hyperideal of $\mathscr{A}$. By Theorem \ref{4}, we conclude that $P_j$ is an $S_j$-hyperideal of $\mathscr{A}$ where $S_j=\mathscr{A} \backslash Q_j$ for all $1 \leq j \leq m$. By Theorem \ref{ejavad}, $P_j$ is an $n$-ary $Q_j$-primary hyperideal of $\mathscr{A}$ for all $1 \leq j \leq m$. Now, we show that $P=\bigcap_{j=1}^mP_j$. Let $p \in P$. So, $g(p,1_\mathscr{A}^{(n-1)}) \in P$. Since $1_\mathscr{A} \notin Q_j$ for all $1 \leq j \leq m$, we get $p \in P_j$ which means $P \subseteq \bigcap_{j=1}^mP_j$. Take any $p \in \bigcap_{j=1}^mP_j$. Therefore, for any $1 \leq j \leq m$,  $g(t_j,p,1_\mathscr{A}^{(n-2)}) \in P$ for some  $ t_j \notin Q_j$. Let $I$ be the hyperideal generated by $X=\{t_1,\ldots,t_m\}$. If $I \cap S=\varnothing$, then $I \subseteq \bigcup_{j=1}^mQ_j$ which implies $I \subseteq Q_j$ for some $1 \leq j \leq m$ by Theorem 5.1 in \cite{mah4}. Since $t_j \in I$, we obtain $t_j \in Q_j$ which is impossible. Then, $I \cap S \neq \varnothing$. Assume that $t \in I \cap S$. Therefore, $t \in f\big(g(a_1,t_1,1_\mathscr{A}^{(n-2)}),\ldots,g(a_m,t_m,1_\mathscr{A}^{(n-2)})\big)$ for some $a_1^m \in \mathscr{A}$. It follows that $g(p,t, 1_\mathscr{A}^{(n-2)}) \in  g\big(p,f\big(g(a_1,t_1,1_\mathscr{A}^{(n-2)}),\ldots,g(a_m,t_m,1_\mathscr{A}^{(n-2)})\big)\big)=f\big(g(a_1,g(p,t_1,1_\mathscr{A}^{(n-2)}),1_\mathscr{A}^{(n-2)}),\ldots,g(a_m,g(p,t_m,1_\mathscr{A}^{(n-2)}),1_\mathscr{A}^{(n-2)}) \big) \subseteq P$. Hence, we have $p \in P$ as  $P$ is an $n$-ary $S$-hyperideal of $\mathscr{A}$ and $t \in S$. This implies that $\bigcap_{j=1}^mP_j \subseteq P$. Thus, we conclude that $P=\bigcap_{j=1}^mP_j$.
\end{proof}
The following result indicates that   the  intersection of an arbitrary family of $n$-ary $S$-hyperideals of $\mathscr{A}$ is also an $n$-ary $S$-hyperideal.
\begin{proposition} \label{intersection}
Let $S$ be an MS of  $\mathscr{A}$ and $\{P_j\}_{j \in \Delta}$ be a nonempty set of the $n$-ary $S$-hyperideals of $\mathscr{A}$. Then, $\bigcap_{j \in \Delta} P_j$ is an $n$-ary $S$-hyperideal  of $\mathscr{A}$.
\end{proposition}
\begin{proof}
Let  $\{P_j\}_{j \in \Delta}$ be a nonempty set of the $n$-ary $S$-hyperideals of $\mathscr{A}$. Assume that $g(p_1^n) \in \bigcap_{j \in \Delta} P_j$ and $p_i \in S$ for some $1 \leq i \leq n$. Since $P_j$ is an $n$-ary $S$-hyperideals for all $j \in \Delta$ and $p_i \in S$, we get $g(p_1^{i-1},1_{\mathscr{A}},p_{i+1}^n) \in P_j$. It follows that $g(p_1^{i-1},1_{\mathscr{A}},p_{i+1}^n) \in \cap_{j \in \Delta} P_j$. Thus, $\bigcap_{j \in \Delta} P_j$ is an $n$-ary $S$-hyperideals of $\mathscr{A}$.
\end{proof}
An element $p \in \mathscr{A}$ is regular if there exist $p_1^{n-1} \in \mathscr{A}$ such that $p=g \big(g(p^{(n)}),p_1^{n-1}\big)$. The set of all regular elements in $\mathscr{A}$ is denoted by $V(\mathscr{A})$. A Krasner $(m,n)$-hyperring $\mathscr{A}$ is called regular, if all of elements
in $\mathscr{A}$ are regular elements,  i.e., $\mathscr{A}=V(\mathscr{A})$. More details can be found in \cite{Hila}. We say that a hyperideal $P$ of $\mathscr{A}$ is $n$-ary semiprime if $g(p^{(n)}) \in P$ for all $p \in \mathscr{A}$ implies that $p \in P$. Moreover, the notion of Krasner $(m,n)$-hyperring of fractions was studied and investigated in \cite{mah5}. 
\begin{theorem} \label{11}
Let $S$ be an MS of $\mathscr{A}$ such that $S \subseteq V(\mathscr{A})$. Then, every $n$-ary $S$-hyperideal of $\mathscr{A}$ is $n$-ary semiprime if and only if $S^{-1}\mathscr{A}$ is regular. 
\end{theorem}
\begin{proof}
$\Longrightarrow$ Assume that $\frac{p}{t}$ is an arbitrary element in $S^{-1}\mathscr{A}$. Put $P=\langle g(p^{(n)}) \rangle$. Let $P \cap S \neq \varnothing$. Then, there exists some $s \in S$ such that $g(s, g(p^{(n)}),1_\mathscr{A}^{(n-2)}) \in S$. This implies that\\ 

$\hspace{2.8cm}\frac{p}{t}=\frac{g \big(g(t^{(n)}),g(s, g(p^{(n)}),1_\mathscr{A}^{(n-2)}),p, 1_\mathscr{A}^{(n-3)}\big)}{g \big(g(t^{(n)}),g(s, g(p^{(n)}),1_\mathscr{A}^{(n-2)}),t ,1_\mathscr{A}^{(n-3)}\big)}$

$\hspace{3.1cm}=\frac{g \big(g(p^{(n)}), g(t^{(n)}),g(s,p,1_\mathscr{A}^{(n-2)}), 1_\mathscr{A}^{(n-3)}\big)}{g \big(g(t^{(n)}),g(s, g(p^{(n)}),1_\mathscr{A}^{(n-2)}),t ,1_\mathscr{A}^{(n-3)}\big)}$

$\hspace{3.1cm}=G \big( G(\frac{p}{t}^{(n)}),\frac{g(t^{(n)})}{g(s,g(p^{(n)}),1_\mathscr{A}^{(n-2)})},\frac{g(s, p,1_\mathscr{A}^{(n-2)})}{t},\frac{1_\mathscr{A}}{1_\mathscr{A}}^{(n-3)}\big).$\\

Hence, $\frac{p}{t}$ is a regular element in $\mathscr{A}$ and so $S^{-1}\mathscr{A}$ is regular. Now, let $P \cap S=\varnothing$. By Theorem \ref{4}, $P^S=\{x \in \mathscr{A} \ \vert \ g(t,x,1_\mathscr{A}^{(n-2)}) \in P\ \text{for some}  \ \ t \in S\}$ is an $n$-ary $S$-hyperideal of $\mathscr{A}$ containing $P$. Hence, $P^S$ is an $n$-ary semiprime hyperideal of $\mathscr{A}$ which implies $p \in P^S$ as $g(p^{(n)}) \in P^S$. Therefore, we obtain $g(r,p,1_\mathscr{A}^{(n-2)}) \in P$ for some $r \in \mathscr{A}$ which means $g(r,p,1_\mathscr{A}^{(n-2)})=g(s,g(p^{(n)}),1_\mathscr{A}^{(n-2)})$ for some $s \in \mathscr{A}$.\\

$\hspace{2.8cm}\frac{p}{t}=\frac{g \big(g(t^{(n-1)},1_\mathscr{A}),r,p,1_\mathscr{A}^{(n-3)}\big)}{g \big(g(t^{(n-1)},1_\mathscr{A}),r,t,1_\mathscr{A}^{(n-3)} \big)}$

$\hspace{3.1cm}=\frac{g \big(g(t^{(n-1)},1_\mathscr{A}),g(r,p,1_\mathscr{A}^{(n-2)}),1_\mathscr{A}^{(n-2)}\big)}{g \big(t^{(n)},r,1_ \mathscr{A}^{(n-2)}\big)}$

$\hspace{03.1cm}=\frac{g \big(g(t^{(n-1)},1_\mathscr{A}),g(s,g(p^{(n)}),1_\mathscr{A}^{(n-2)}),1_\mathscr{A}^{(n-2)}\big)}{g \big(g(t^{(n)}),r, 1_\mathscr{A}^{(n-2)}\big)}$

$\hspace{3.1cm}=\frac{g \big(g(t^{(n-1)},s),g(p^{(n)}),1_\mathscr{A}^{(n-2)}\big)}{g \big ( g(t^{(n)}),r, 1_\mathscr{A}^{(n-2)} \big)}$

$\hspace{3.1cm}=G \big(G(\frac{p}{t}^{(n)}), \frac{g(t^{(n-1)},s)}{r},\frac{1_\mathscr{A}}{1_\mathscr{A}}^{(n-2)}\big).$\\

Then, we conclude that $\frac{p}{t}$ is a regular element in $S^{-1}\mathscr{A}$ and so  $S^{-1}\mathscr{A}$ is regular. 

$\Longleftarrow$  Let  $P$ be an $n$-ary $S$-hyperideal of $\mathscr{A}$  and $g(p^{(n)}) \in P$ for $p \in \mathscr{A}$. Then, there exist $\frac{p_1}{t_1}, \ldots, \frac{p_{n-1}}{t_{n-1}} \in S^{-1}\mathscr{A}$ such that $\frac{p}{1_\mathscr{A}}=G \big(G(\frac{p}{1_\mathscr{A}}^{(n)}), \frac{p_1}{t_1}, \ldots, \frac{p_{n-1}}{t_{n-1}} \big)$  
  as $S^{-1}\mathscr{A}$ is regular and $\frac{p}{1_\mathscr{A}} \in S^{-1}\mathscr{A}$. Since $G(\frac{p}{1_\mathscr{A}}^{(n)})=\frac{g(p^{(n)})}{g(1_\mathscr{A}^{(n)})} \in S^{-1}P$, we obtain $\frac{p}{1_\mathscr{A}}  \in S^{-1}P$. This implies that $g(r,p,1_\mathscr{A}^{(n-2)}) \in P$ for some $r \in S$. Then, we conclude that $p \in P$ as $P$ is an $n$-ary $S$-hyperideal of $\mathscr{A}$. Consequently, every $n$-ary $S$-hyperideal of $\mathscr{A}$ is $n$-ary semiprime. 
\end{proof}

An element $p$ in $\mathscr{A}$ is invertible if there exists $q \in \mathscr{A}$ such that $g(p,q,1_{\mathscr{A}}^{(n-2)})=1_{\mathscr{A}}$.  By  $\mathbf{U}(\mathscr{A})$, we mean all invertible elements in $\mathscr{A}$ \cite{sorc1}.
\begin{proposition} \label{8}
Let $S$ be an MS of $\mathscr{A}$ containing $1_\mathscr{A}$. Then, any proper hyperideal of $\mathscr{A}$ is an $n$-ary $S$-hyperideal if and only if $S \subseteq \mathbf{U}(\mathscr{A})$.
\end{proposition}
\begin{proof}
$\Longrightarrow$ Let each proper hyperideal of $\mathscr{A}$ be an $n$-ary $S$-hyperideal. Assume that $ s \in S \backslash \mathbf{U}(\mathscr{A})$. Then, $\langle s \rangle$ is a proper hyperideal of $\mathscr{A}$ and so it is an $n$-ary $S$-hyperideal of $\mathscr{A}$ by the hypothesis. Since $g(s,1_\mathscr{A}^{(n-1)})=s \in \langle s \rangle$, we obtain $1_\mathscr{A} \in  \langle s \rangle$ which is impossible. Hence, 
$S \subseteq \mathbf{U}(\mathscr{A})$. 

$\Longleftarrow$ Let $P$ be an arbitrary hyperideal of $\mathscr{A}$ and $g(p_1^n) \in P$ for $p_1^n \in \mathscr{A}$ and $p_i \in S$ for some $1 \leq i \leq n$. Since $P$ is a hyperideal of $\mathscr{A}$ and  $S \subseteq \mathbf{U}(\mathscr{A})$, we get $g(p_1^{i-1},g(p_i,p_i^{-1},1_\mathscr{A}^{(n-2)}),p_{i+1}^n)=g(p_i^{-1},g(p_1^n),1_\mathscr{A}^{(n-2)}) \in P$ which implies  $g(p_1^{i-1},1_\mathscr{A},p_{i+1}^n) \in P$. Thus, any proper hyperideal of $\mathscr{A}$ is an $n$-ary $S$-hyperideal.
\end{proof}
A  Krasner $(m,n)$-hyperring $\mathscr{A}$ is said to be an $n$-ary  hyperintegral domain, if $\mathscr{A}$ is commutative and $g(p_1^n)=0$ implies that $p_i=0$ for some $1 \leq i \leq n$ \cite{sorc1}.
\begin{theorem}\label{9}
Let $S$ be an MS of $\mathscr{A}$ such that $S=\{x \in \mathscr{A} \ \vert \ \frac{x}{1_{\mathscr{A}} } \in \mathbf{U}(S^{-1}\mathscr{A})\}$. Then, $\mathscr{A}$ is an $n$-ary hyperintegral domain and $S=\mathscr{A}-\{0\}$ if and only if   there is no $n$-ary $S$-hyperideal    in $\mathscr{A}$, other than $\langle 0 \rangle$.
\end{theorem}
\begin{proof}
 $\Longrightarrow$ Assume that $\mathscr{A}$ is an $n$-ary hyperintegral domain and $S=\mathscr{A}-\{0\}$. Let $g(p_1^n) \in \langle 0 \rangle$ for $p_1^n \in \mathscr{A}$ and $p_i \in S$ for some $1 \leq i \leq n$. Since  $\mathscr{A}$ is an $n$-ary hyperintegral domain,  $g(g(p_1^{i-1},1_\mathscr{A},p_{i+1}^n),p_i,1_\mathscr{A}^{(n-2)}) =0$ and $ p_i \neq 0$, we get $g(p_1^{i-1},1_\mathscr{A},p_{i+1}^n)=0$ which means $\langle 0 \rangle$ is an $n$-ary $S$-hyperideal of $\mathscr{A}$. Now, let $P \neq \{0\}$ be a  hyperideal of $\mathscr{A}$. Then, $P \cap S \neq \varnothing$ which means $P$ is not an $n$-ary $S$-hyperideal of $\mathscr{A}$ by Theorem \ref{1} (1). Thus, there is no $n$-ary $S$-hyperideal    in $\mathscr{A}$, other than $\langle 0 \rangle$.

$\Longleftarrow$ Let  $a \notin \mathscr{A} \backslash S$ such that $a \neq 0$. By the hypothesis, we conclude that $\langle a \rangle \cap S=\varnothing$. Hence, ${\langle a \rangle}^S$ is   $n$-ary $S$-hyperideal of $\mathscr{A}$ by Theorem \ref{4}. Since $\langle 0 \rangle$ is the only $n$-ary $S$-hyperideal of $\mathscr{A}$,  we obtain ${\langle a \rangle}^S=\langle 0 \rangle$ which is impossible. Thus, $S=\mathscr{A}-\{0\}$. Now, we show that $\mathscr{A}$ is an $n$-ary hyperintegral domain. Since $0 \notin S$, there exists an $n$-ary prime hyperideal $Q$ of $\mathscr{A}$ disjoint from $S$ by Theorem 4.21 in \cite{sorc1}. Therefore, $Q$ is an $n$-ary $S$-hyperideal of $\mathscr{A}$ by Proposition \ref{2}. Since $\langle 0 \rangle$ is the only $n$-ary $S$-hyperideal of $\mathscr{A}$, we have $Q=\langle 0 \rangle$ and so $\langle 0 \rangle$ is an $n$-ary prime hyperideal of $\mathscr{A}$. This implies that $\mathscr{A}$ is an $n$-ary hyperintegral domain by Theorem 4.6 in \cite{sorc1}. 
\end{proof}
\begin{theorem} \label{10}
Let $Q$ be a proper hyperideal of $\mathscr{A}$ and $S=f(Q,1_\mathscr{A},0^{(m-2)})$. Then, every proper hyperideal of $\mathscr{A}$ containing $Q$  is an $n$-ary $S$-hyperideal of $\mathscr{A}$. Furthermore, if $P$ is  a maximal   $n$-ary $S$-hyperideal  of $\mathscr{A}$ and $Q$ is contained in $J_{(m,n)}(\mathscr{A})$, then $Q \subseteq P$. 
\end{theorem}
\begin{proof}
Suppose that  $P$ is an arbitrary  hyperideal of $\mathscr{A}$ such that $Q \subseteq P$ and $S=f(Q,1_\mathscr{A},0^{(m-2)})$. Assume that $g(p_1^n) \in P$ for $p_1^n \in \mathscr{A}$ and $p_i \in S$ for some $1 \leq i \leq n$. Therefore, $p_i \in f(q,1_\mathscr{A},0^{(m-2)})$ for some $q \in Q$ and so $g(p_1^n) \in g \big(p_1^{i-1},f(q,1_\mathscr{A},0^{(m-2)}), p_{i+1}^n \big)=f \big(g(p_1^{i-1},q,p_{i+1}^n),g(p_1^{i-1},1_\mathscr{A},p_{i+1}^n),0^{(m-2)} \big)$. This implies that $g(p_1^{i-1},1_\mathscr{A},p_{i+1}^n) \in f(-g(p_1^{i-1},q,p_{i+1}^n), g(p_1^n),0^{(m-2)})$. Since $P$ contains $f(-g(p_1^{i-1},q,p_{i+1}^n), g(p_1^n),0^{(m-2)})$, we obtain $g(p_1^{i-1},1_\mathscr{A},p_{i+1}^n) \in P$. Hence,  $P$ is an $n$-ary $S$-hyperideal of $\mathscr{A}$. 

Now, suppose that  the hyperideal $Q$ is contained in $J_{(m,n)}(\mathscr{A})$. Then, by Theorem 4.14 in \cite{sorc1} we conclude that $S =f(Q,1_\mathscr{A},0^{(m-2)})\subseteq {\bf U}(\mathscr{A})$. Therefore, each proper hyperideal of $\mathscr{A}$ is an $n$-ary $S$-hyperideal by Proposition \ref{8}. Let $Q \nsubseteq P$. Then, there exist some $q \in Q$ such that $q \notin P$. Clearly, $f(P,\langle q \rangle,0^{(m-2)}) \subsetneqq  P$. By the hypothesis,  we conclude that  $f(P,\langle q \rangle,0^{(m-2)})=\mathscr{A}$ which implies $1_\mathscr{A} \in f(p,g(a,q,1_\mathscr{A}^{(n-2)}),0^{(m-2)})$ for some $p \in P$ and $a \in \mathscr{A}$. Hence, we get $p \in f(1_\mathscr{A},-g(a,q,1_\mathscr{A}^{(n-2)}) \subseteq {\bf U}(\mathscr{A})$ which is impossible. Thus, $Q \subseteq P$.
\end{proof}
\begin{theorem} \label{12}
Let $P$ be a hyperideal of $\mathscr{A}$ such that for every $p \in P$, the intersection   of all
minimal $n$-ary prime hyperideals containing $p$ is contained in $P$. Then, $P$ is an $n$-ary $S$-hyperideal of $\mathscr{A}$ where $S=\mathscr{A} \backslash \cup_{Q \in Min(\mathscr{A})} Q $.
\end{theorem}
\begin{proof}
Let $g(p_1^n) \in P$ for $p_1^n \in \mathscr{A}$ and $p_i \in S$ for some $1 \leq i \leq n$. Assume that $Q$ is the intersection   of all
minimal $n$-ary prime hyperideals containing $g(p_1^n)$. By the hypothesis, we obtain $Q \subseteq P$.  From $p_i \in S$ for some $1 \leq i \leq n$, it follows that $p_i$ is not in any minimal $n$-ary prime hyperideal of $\mathscr{A}$ and so $p_i \notin Q$. Since $g(g(p_1^{i-1},1_\mathscr{A},p_{i+1}^n),p_i,1_ \mathscr{A}^{(n-2)}) \in Q$  and $p_i \notin Q$, we conclude that $g(p_1^{i-1},1_\mathscr{A},p_{i+1}^n)$ is in every minimal $n$-ary prime hyperideal containing $g(p_1^n)$ and so $g(p_1^{i-1},1_\mathscr{A},p_{i+1}^n) \in Q \subseteq P$. Consequently, $P$ is an $n$-ary $S$-hyperideal of $\mathscr{A}$.
\end{proof}
  Here, we present a version of the Prime Avoidance Lemma adapted to the context of $S$-hyperideals.
\begin{theorem} \label{avoidance}
Let $S$ be an MS of  $\mathscr{A}$ containing $1_{\mathscr{A}}$ and $P,P_1^n$ be hyperideals of $\mathscr{A}$ such that $P \subseteq \bigcup_{j=1}^n P_j$ but no $P_j$ can be omitted from the union. If $P_t$ is an $n$-ary $S$-hyperideal of $\mathscr{A}$ for some $1 \leq t \leq n$ and $P_j \cap S \neq \varnothing$ for all $j \in \{1,\dots,t\!-\!\!1,\hat{t},t\!+\!1,\ldots,n\}$, then $P \subseteq P_t$. 
\end{theorem}
\begin{proof}
By the assumption, we get $P \nsubseteq \bigcup_{\substack{1 \leq j \leq i \\ j \neq t}} P_j$. Then, there exists some $p \in P$ such that $p \notin \bigcup_{\substack{1 \leq j \leq i \\ j \neq t}} P_j$. So, $p \in P_t$. Assume that $a \in P \cap (\bigcap_{\substack{1 \leq j \leq i \\ j \neq t}} P_j)$. Therefore, we have $f(a,p,0^{(m-2)}) \subseteq P$. Let  $f(a,p,0^{(m-2)}) \subseteq \bigcup_{\substack{1 \leq j \leq i \\ j \neq t}} P_j$. Then  $f(p,a,-a,0^{(m-3)})=f(f(a,p,0^{(m-2)}),-a,0^{(m-2)}) \subseteq \bigcup_{\substack{1 \leq j \leq i \\ j \neq t}} P_j$. So,  we obtain $p \in  \bigcup_{\substack{1 \leq j \leq i \\ j \neq t}} P_j$ which is impossible.  Thus, we conclude that $f(a,p,0^{(m-2)}) \nsubseteq \bigcup_{\substack{1 \leq j \leq i \\ j \neq t}} P_j$ which implies $f(a,p,0^{(m-2)}) \subseteq P_t$ as $f(a,p,0^{(m-2)}) \subseteq P$. Hence, we get $a \in P_t$  as $p \in P_t$. This means that $P \cap (\bigcap_{\substack{1 \leq j \leq i \\ j \neq t}} P_j) \subseteq P_t$ which implies $g(P_1^{t-1},P,P_{t+1}^n) \subseteq P_t$. By the hypothesis, we have $s_j \in P_j \cap S$ for all  $j \in \{1,\dots,t\!-\!\!1,\hat{t},t\!+\!1,\ldots,n\}$. Then, we have $g(g(s_1^{t-1},1_{\mathscr{A}},s_{t+1}^n),x,1_{\mathscr{A}})=g(s_1^{t-1},x,s_{t+1}^n) \in P_t$ for all $x \in P$. Since $P_t$ is an $n$-ary $S$-hyperideal of $\mathscr{A}$ and $ g(s_1^{t-1},1_{\mathscr{A}},s_{t+1}^n)\in S$, we have $x=g(x,1_{\mathscr{A}}^{(n-1)}) \in P_t$ which indicates $P \subseteq P_t$.
\end{proof}
\section{Stability of $n$-ary $S$-hyperideals}
In this section, we characterize the behavior of $n$-ary $S$-hyperideals under various hyperring-theoretic constructions.

Let $\psi$ from $\mathscr{A}_1$ to $\mathscr{A}_2$ be a mapping  where $(\mathscr{A}_1, f_1, g_1)$ and $(\mathscr{A}_2, f_2, g_2)$ are two commutative Krasner $(m, n)$-hyperrings. Recall from \cite{d1} that $\psi$ is  a homomorphism if for any $u^m _1, v^n_ 1 \in \mathscr{A}_1$ we have
\begin{itemize}
\item[\rm{(i)}]~$\psi(f_1(u_1^m)) = f_2(\psi(u_1),\ldots, \psi(u_m)),$
\item[\rm{(ii)}]~$\psi(g_1(v_1^n)) = g_2(\psi(v_1),\ldots,\psi(v_n)),$
\item[\rm{(iii)}]~$\psi(1_{\mathscr{A}_1})=1_{\mathscr{A}_2}.$
\end{itemize}
\begin{theorem} \label{homo}
Let $\psi$ from $\mathscr{A}_1$ to $\mathscr{A}_2$ be a homomorphism where $(\mathscr{A}_1, f_1, g_1)$ and $(\mathscr{A}_2, f_2, g_2)$ are  commutative Krasner $(m,n)$-hyperrings and let $S$ be an MS of  $\mathscr{A}_1$. Then, the following assertions  hold. 
\begin{itemize}
\item[\rm{(ii)}]~ The inverse image of every $n$-ary $\psi(S)$-hyperideal of $\mathscr{A}_2$ is an $n$-ary $S$-hyperideal of $\mathscr{A}_1$.
\item[\rm{(ii)}]~ If $\psi$ is an epimorphism, then the image of every $n$-ary $S$-hyperideal of $\mathscr{A}_1$ containing $ker (\psi)$ is an $n$-ary $\psi(S)$-hyperideal of $\mathscr{A}_2$.
\end{itemize} 
\end{theorem}
\begin{proof}
(i) Let $Q$ be an $n$-ary $\psi(S)$-hyperideal of $\mathscr{A}_2$. Assume that  $g_1(p_1^n) \in \psi^{-1}(Q)$ for $p_1^n \in \mathscr{A}_1$ and $p_i \in S$ for some $1 \leq i \leq n$. This implies that $g_2 \big(\psi(p_1),\ldots,\psi(p_n) \big)=\psi(g_1(p_1^n)) \in Q$. Since $Q$ is an $n$-ary $\psi(S)$-hyperideal of $\mathscr{A}_2$ and $\psi(p_i) \in \psi(S)$, we obtain $\psi(g_1(p_1^{i-1},1_{\mathscr{A}_1},p_{i+1}^n)=g_2 \big(\psi(p_1),\ldots,\psi(p_{i-1}),1_{\mathscr{A}_2},\psi(p_{i+1}),\ldots,\psi(p_n)\big) \in Q$ which means $g_1(p_1^{i-1},1_{\mathscr{A}_1},p_{i+1}^n) \in \psi^{-1}(Q)$. Consequently, $\psi^{-1}(Q)$ is an $n$-ary $S$-hyperideal of $\mathscr{A}_1$.

(ii) Let $P$ be an $n$-ary $S$-hyperideal of $\mathscr{A}_1$ containing $ker (\psi)$ where $\psi$ is an epimorphism.  Suppose that $g_2(q_1^n) \in \psi(P)$ for $q_1^n \in \mathscr{A}_2$ and $q_i \in \psi(S)$ for some $1 \leq i \leq n$. By the hypothesis, we have $\psi(p_i) =q_i$ for some $ p_i \in S$ and there exists $p_t \in \mathscr{A}_1$ for every $1 \leq t \leq i$ and $i \leq t \leq n$ such that $\psi(p_t)=q_t$. Therefore,  we get $\psi(g_1(p_1^n))=g_2(\psi(p_1),\ldots,\psi(p_n))=g_2(q_1^n) \in \psi(P)$. Since $P$ contains $ker (\psi)$, we conclude that $g_1(p_1^n) \in P$. This implies that $g_1(p_1^{i-1},1_{\mathscr{A}_1},p_{i+1}^n) \in P$ as $P$ is  an $n$-ary $S$-hyperideal of $\mathscr{A}_1$ and $p_i \in S$. Then,  $g_2 (q_1^{i-1},1_{\mathscr{A}_2},q_{i+1}^n)=g_2 \big(\psi(p_1),\ldots,\psi(p_{i-1}),1_{\mathscr{A}_2},\psi(p_{i+1}),\ldots,\psi(p_n)\big)=\psi(g_1(p_1^{i-1},1_{\mathscr{A}_1},p_{i+1}^n)) \in \psi(P)$. Thus, $\psi(P)$ is an $n$-ary $\psi(S)$-hyperideal of $\mathscr{A}_2$.
\end{proof}
\begin{theorem}
Let $P \subseteq Q$ be hyperideals of $\mathscr{A}$ and $S$ be an MS of  $\mathscr{A}$. Then, $Q$ is an $n$-ary $S$-hyperideal of $\mathscr{A}$ if and only if $Q/P$ is an $n$-ary $T$-hyperideal of $\mathscr{A}/P$ where $T=\{f(r,P,0^{(m-2)}) \ \vert \ r \in S \}$.
\end{theorem}
\begin{proof}
Consider the map $\pi :\mathscr{A} \longrightarrow \mathscr{A}/P$, defined by $x \longrightarrow f(x,P,0^{(m-2)})$. This map is an epimorphism by Theorem 3.2 in \cite{sorc1}. By
using Theorem \ref{homo}, the claim can be proved.
\end{proof}
\begin{theorem}
Let $S,S^{\prime}$ be two MS$^,s$ of  $\mathscr{A}$ and $P$ be a proper hyperideal of $\mathscr{A}$ disjoint from $S^{\prime}$ such that $1_\mathscr{A} \in S^{\prime}$.  
\begin{itemize} 
\item[\rm{{\bf (1)}}]~ 
If $P$ is an $n$-ary $S$-hyperideal of $\mathscr{A}$, then ${S^{\prime}}^{-1} P$ is an $n$-ary ${S^{\prime}}^{-1} S$-hyperideal of ${S^{\prime}}^{-1} \mathscr{A}$ where ${S^{\prime}}^{-1} S=\{\frac{s}{r} \ \vert s \in S, r \in S^{\prime}\}$.
\item[\rm{{\bf (2)}}]~ every  proper hyperideal of $S^{-1} \mathscr{A}$ is an $n$-ary $S^{-1} S$-hyperideal.
\end{itemize}
\end{theorem}
\begin{proof}
(1) Let $P$ be an $n$-ary $S$-hyperideal of $\mathscr{A}$ disjoint from  $S^{\prime}$. Clearly,  ${S^{\prime}}^{-1} P \cap {S^{\prime}}^{-1} S=\varnothing$. Assume that $G(\frac{p_1}{r_1},\ldots,\frac{p_n}{r_n}) \in {S^{\prime}}^{-1} P$ for $\frac{p_1}{r_1},\ldots,\frac{p_n}{r_n} \in {S^{\prime}}^{-1} \mathscr{A}$ and $\frac{p_i}{r_i} \in {S^{\prime}}^{-1} S$ for some $1 \leq i \leq n$. From $\frac{g(p_1^n)}{g(r_1^n)} \in {S^{\prime}}^{-1} P$, it follows that $,g(r,g(p_1^n),1_\mathscr{A}^{(n-2)}) \in P$ for some $r \in S^{\prime}$. Since $g \big(g(p_1^{i-1},r,p_{i+1}^n) ,p_i,1_\mathscr{A}^{(n-2)}\big) \in P$, $p_i \in S$ and $P$ is an $n$-ary $S$-hyperideal of $\mathscr{A}$,  we get $g(p_1^{i-1},r,p_{i+1}^n) \in P$ and so 
$G(\frac{p_1}{r_1},\ldots,\frac{p_{i-1}}{r_{i-1}}, \frac{1_\mathscr{A}}{1_\mathscr{A}}, \frac{p_{i+1}}{r_{i+1}},\frac{p_n}{r_n})=\frac{g(p_1^{i-1},1_\mathscr{A},p_{i+1}^n)}{g(r_1^{i-1},1_\mathscr{A},r_{i+1}^n)}=\frac{g(p_1^{i-1},r,p_{i+1}^n) }{g(r_1^{i-1},r,r_{i+1}^n)} \in {S^{\prime}}^{-1} P$. Hence, ${S^{\prime}}^{-1} P$ is an $n$-ary ${S^{\prime}}^{-1} S$-hyperideal of ${S^{\prime}}^{-1} \mathscr{A}$.

(2) Take $S=S^{\prime}$. In this case, $S^{-1}S \subseteq {\bf U}(S^{-1}\mathscr{A})$. Therefore, every  proper hyperideal of $S^{-1} \mathscr{A}$ is an $n$-ary $S^{-1} S$-hyperideal by Proposition \ref{8}.
\end{proof}

Suppose that  $(\mathscr{A}_1, f_1, g_1)$ and $(\mathscr{A}_2, f_2, g_2)$ be  commutative Krasner $(m,n)$-hyperrings. Then, $(\mathscr{A}_1 \times \mathscr{A}_2, f _1 \times f_2 ,g_1 \times g_2 )$ is a Krasner $(m, n)$-hyperring such that  $m$-ary hyperoperation $f _1 \times f_2 $ and $n$-ary operation $g_1 \times g_2$ are defined as follows:

$\hspace{1cm} f_1 \times f_2((r_{1}, t_{1}),\ldots,(r_m,t_m)) = \{(r,t) \ \vert \ \ t \in f_1(r_1^m), r \in f_2(t_1^m) \},$

$\hspace{1cm} g_1 \times g_2 ((p_1,q_1),\ldots,(p_n,q_n)) =(g_1(p_1^n),g_2(q_1^n)) $,\\
for $r_1^m,p_1^n \in \mathscr{A}_1$ and $t_1^m,q_1^n \in \mathscr{A}_2$ \cite{mah2}. Now, we present a characterization  of the concept of  $n$-ary $S$-hyperideals on cartesian product of Krasner $(m, n)$-hyperrings.

\begin{theorem} \label{cart2}
Suppose that $P_1^k$ are hyperideals of $\mathscr{A}_1^k$, respectively, such that $(\mathscr{A}_1, f_1, g_1), \ldots, (\mathscr{A}_k, f_k, g_k)$ are commutative Krasner $(m,n)$-hyperrings and $S_j$ is an  $n$-ary MS of $\mathscr{A}_j$ for every $1 \leq j \leq k$. Then, $P_1 \times \cdots \times P_k$ is an $n$-ary $(S_1 \times \cdots \times S_k)$-hyperideal of $\mathscr{A}_1 \times \cdots \times \mathscr{A}_k$ if and only if $P_j$ is an $n$-ary $S_j$-hyperideal of $\mathscr{A}_j$ for every $1 \leq j \leq k$.
\end{theorem}
\begin{proof}
$\Longrightarrow$ Let $P_1 \times \cdots \times P_k$ be an $n$-ary $(S_1 \times \cdots \times S_k)$-hyperideal of $\mathscr{A}_1 \times \cdots \times \mathscr{A}_k$. Take any $1 \leq j \leq k$. Assume that $g_j(p_1^n) \in P_j$ for $p_1^n \in \mathscr{A}_j$ and $p_i \in S_j$ for some $1 \leq i \leq n$. Put  $q_i=(1_{\mathscr{A}_1},\ldots,1_{\mathscr{A}_{j-1}},p_i,1_{\mathscr{A}_{j+1}},\ldots,1_{\mathscr{A}_k})$ and $q_l=(0_{\mathscr{A}_1},\ldots,0_{\mathscr{A}_{j-1}},p_l,0_{\mathscr{A}_{j+1}},\ldots,0_{\mathscr{A}_k})$ for all $1 \leq l \leq i$ and $i \leq l \leq n$. This follows that $g_1 \times \ldots \times  g_k (q_1^n) \in P_1 \times \cdots \times P_k$. This implies that $g_1 \times \ldots \times  g_k (q_1^{i-1},1_{\mathscr{A}_1 \times \cdots \times \mathscr{A}_k},q_{i+1}^n) \in P_1 \times \cdots \times P_k$ as $P_1 \times \cdots \times P_k$ is an $n$-ary $(S_1 \times \cdots \times S_k)$-hyperideal of $\mathscr{A}_1 \times \cdots \times \mathscr{A}_k$ and $q_i \in S_1 \times \cdots \times S_k$. Therefore,  we obtain $(0_{\mathscr{A}_1},\ldots,0_{\mathscr{A}_{j-1}},g_j(p_1^{i-1},1_{\mathscr{A}_j},p_{i+1}^n),0_{\mathscr{A}_{j+1}},\ldots,0_{\mathscr{A}_k}) \in P_1 \times \cdots \times P_k$ and so $g_j(p_1^{i-1},1_{\mathscr{A}_j},p_{i+1}^n) \in P_j$. Thus, $P_j$ is an $n$-ary $S_j$-hyperideal of $\mathscr{A}_j$ for every $1 \leq j \leq k$.

$\Longleftarrow$ Assume that $g_1 \times \cdots \times g_k \big((p_{11},\ldots,p_{1k}),\ldots,(p_{n1},\ldots,p_{nk} )\big) \in P_1 \times \cdots \times P_k$ for $(p_{11},\ldots,p_{1k}),\ldots,(p_{n1},\ldots,p_{nk} ) \in \mathscr{A}_1 \times \cdots \times \mathscr{A}_k$ and $(p_{i1},\ldots,p_{ik} ) \in S_1 \times \cdots \times S_k$ for some $1 \leq i \leq n$. Therefore,  $\big(g_1(p_{11}^{n1}),\ldots,g_k(p_{1k}^{nk})\big) \in P_1 \times \cdots \times P_k$ which means $g_j(p_{1j}^{nj}) \in P_j$ for every $1 \leq j \leq k$. Hence, we obtain $g_j(p_{1j}^{{(i-1)}j},1_{\mathscr{A}_j}, p_{(i+1)j}^{nj}) \in P_j$ as $P_j$ is an $n$-ary $S_j$-hyperideal of $\mathscr{A}_j$ for every $1 \leq j \leq k$ and $p_{ij} \in S_j$. Then, we conclude that $\big( g_1(p_{11}^{{(i-1)}1},1_{\mathscr{A}_1}, p_{(i+1)1}^{n1}),\ldots,g_k(p_{1k}^{{(i-1)}k},1_{\mathscr{A}_k}, p_{(i+1)k}^{nk})   \big) \in P_1 \times \cdots \times P_k$ which implies $g_1 \times \cdots \times g_k \big((p_{11},\ldots,p_{1k}),\ldots,(p_{{(i-1)}1},\ldots,p_{(i-1)k}),(1_{\mathscr{A}_1},\ldots,1_{\mathscr{A}_k}),\break(p_{{(i+1)}1},\ldots,p_{(i+1)k}),\ldots,(p_{n1},\ldots,p_{nk} )\big) \in P_1 \times \cdots \times P_k$ Consequently, $P_1 \times \cdots \times P_k$ is an $n$-ary $(S_1 \times \cdots \times S_k)$-hyperideal of $\mathscr{A}_1 \times \cdots \times \mathscr{A}_k$.
\end{proof}

\section{conclusion}
The main objective of this paper was to introduce  and investigate  the concept of $n$-ary $S$-hyperideals, shedding new light on the classification of hyperideals  by means of  an $n$-ary multiplicative subset $S$ of a Krasner $(m,n)$-hyperring $\mathscr{A}$.
We comprehensively examined the properties of $n$-ary $S$-hyperideals and analyzed their connections with other classes of hyperideals. Specifically, we showed that every maximal $n$-ary $S$-hyperideal of $\mathscr{A}$ is $n$-ary prime. Furthermore, we demonstrated that the existence of an $n$-ary multiplicative subset $S$ of $A$ is guaranteed for every hyperideal $P$ of $A$ such that $P$ is an $n$-ary $S$-hyperideal. Moreover, the smallest $n$-ary $S$-hyperideal of $\mathscr{A}$ containing an arbitrary hyperideal was presented. While $S$-hyperideals and primary hyperideals represent different concepts, we proved that  $n$-ary $S$-hyperideals and $n$-ary $Q$-primary hyperideals are equivalent when $S$ is chosen as the complement of the minimal prime hyperideal $Q$. We presented  a modification of the Prime Avoidance Lemma in terms of $S$-hyperideals. Finally, we  investigated how $S$-hyperideals behave under various hyperring-theoretic constructions.
\section{Future work}
\begin{definition}
Let $S$ be an MS of $\mathscr{A}$. A proper hyperideal $P$ of $\mathscr{A}$ is called an $n$-ary  $S_r$-hyperideal if  $g(p_1^n) \in P$ for   $p_1^n \in  \mathscr{A}$ and $p_i \in S$ for some $1 \leq i \leq n$ imply that $g(p_1^{i-1},1_\mathscr{A},p_{i+1}^n) \in {\bf r}(P)$. 
\end{definition}

\end{document}